\documentclass{amsart}

\usepackage{amsthm,amssymb,verbatim}
\usepackage{enumerate}

\usepackage[alphabetic]{amsrefs}

\newcommand{\mass}[1]{|#1|}
\newcommand{\bmass}[1]{|\delta #1|}
\newcommand{\LL}{\mathcal L} 

\newcommand{\vv}{\mathbf{v}}
\newcommand{\kk}{\kappa}
\newcommand{\Div}{\operatorname{div}}

 \newcommand{\RR}{\mathbf{R}}  
 \newcommand{\dist}{\operatorname{dist}}
 
 \newcommand{\area}{\operatorname{area}}
 \newcommand{\eps}{\epsilon}
 \newcommand{\Tan}{\operatorname{Tan}}

 \newcommand{\spt}{\operatorname{spt}}

\input epsf
\def\begfig {
\begin{figure}
\small }
\def\endfig {
\normalsize
\end{figure}
}

\swapnumbers 

    \newtheorem{theorem}    {Theorem}       [section]
    
    \newtheorem{corollary}  [theorem]     {Corollary}
    \newtheorem{proposition}       [theorem]       {Proposition}

    \theoremstyle{definition}
    \newtheorem{definition}  [theorem] {Definition}
    \theoremstyle{definition}
    \newtheorem{remark}   [theorem]       {Remark}

\begin{document}


%
%
%
%
%
%
%

\renewcommand{\thesubsection}{\thetheorem}

\title[Isoperimetric Inequalities]{Which Ambient Spaces Admit \\
         Isoperimetric Inequalities \\
         for Submanifolds?}
\author{Brian White}
\thanks{This research was supported in part by the National Science Foundation 
  grants~DMS-0406209 and~DMS-0707126.}
\email{white@math.stanford.edu}
\date{February 23, 2008.  Revised December 16,  2008}

\begin{abstract}
 We give simple  conditions on an ambient manifold 
 that are necessary and sufficient for isoperimetric inequalities to hold.
\end{abstract}


\maketitle

\section{Introduction}\label{intro}

Let $N$ be a compact $(n+1)$-dimensional Riemannian manifold with mean convex boundary.
Can one bound the $n$-dimensional area of a minimal hypersurface in $N$ in terms of the $(n-1)$-dimensional area of
its boundary?   The absence of any closed minimal hypersurface in $N$ is certainly
a necessary condition, since such a hypersurface would contradict any
such bound.  In this paper, we show that this necessary condition is also sufficient.  Indeed, we
show (Theorem~\ref{HypersurfaceCase})
that
the absence of such a hypersurface implies the existence of a $c=c_N<\infty$ such that
\begin{equation}\label{LinearInequality}
   |M| \le c \left(  |\partial M| + \int_M|H| \right)
\end{equation}
for every $n$-dimensional variety $M$ in $N$, where $H(x)$ is the mean curvature of $M$ at $x$.  We also prove an analogous but weaker
result (Theorem~\ref{GeneralCodimension})
about isoperimetric inequalities for surfaces of any codimension.
(For surfaces in compact subsets of Euclidean space,
 the inequality~\eqref{LinearInequality} has a 
  simple, well-known proof: see Theorem~\ref{EasyVersion}.)

The general codimension theorem involves varifolds in an essential way, but the proof uses
only the most
elementary facts about them.  Indeed, the theorem is a rather
trivial consequence of the basic definitions.
For readers not familiar with varifolds, we have included all the necessary
background (with proofs)  in an appendix.

The hypersurface theorem is a consequence of the general theorem together with two
facts about mean curvature flow: (1) a hypersurface moving by mean curvature flow cannot
bump into a minimal variety, and (2) a mean convex hypersurface moving by mean curvature
flow either vanishes in finite time or converges as $t\to\infty$ to a minimal hypersurface with a very
 small singular set.  These facts are not elementary, but otherwise this paper is for the most part self-contained.
 
 After proving the main theorems in Section~\ref{main}, we give two applications
 in Section~\ref{examples}.  Section~\ref{disks} 
 discusses the special case of minimal embedded disks in 
$3$-manifolds.   Section~\ref{nonsmooth} explains how to extend
the hypersurface results to ambient manifolds with piecewise smooth
boundaries. Section~\ref{NonlinearSection} discusses the nonlinear isoperimetric
inequality obtained by replacing $|M|$ with  $|M|^{(n-1)/n}$ in~\eqref{LinearInequality}.

\section{The Main Results}\label{main}

Suppose that $N$ is a compact Riemannian manifold with 
smooth boundary.  We say that $N$ is {\it mean convex}
provided that the mean curvature vector at each point
of $\partial N$ is a nonnegative multiple of the inward-pointing
unit normal.

\newcommand{\LLL}{\LL_{\text{$m$-rec}}}

We begin with the main result for hypersurfaces.  The reader
may wish to ignore the ``furthermore\dots" assertion until it
is used in Section~\ref{examples}.

\begin{theorem}\label{HypersurfaceCase}
Suppose that $N$ is a compact, connected, mean-convex
 $(n+1)$-dimensional Riemannian manifold
with smooth, nonempty boundary, 
and that $n<7$. The following are equivalent:
\begin{enumerate}[\upshape (a)]
\item\label{hypersurface1}
 $N$ contains no smooth, closed, embedded minimal hypersurface. 
\item\label{hypersurface2}
 There is an increasing function $\phi: \RR^+\to \RR^+$ with $\phi(0)=0$ such that
\[
   |M| \le \phi(|\partial M|)
\]
for every smooth $n$-dimensional minimal surface $M$ in $N$.
\item\label{hypersurface3}
 There is a constant $c< \infty$ such that if $M$ is a smooth
$n$-dimensional manifold
in $N$, then
\[
     |M| \le c \left( |\partial M| + \int_M |H|\,dA\right).
\]
\item\label{hypersurface4}
 There is a constant $c<\infty$ such that if $V$ is an $n$-dimensional varifold in $N$, then
\[
    \mass{V} \le c \,\bmass{V}.
\]
where $\mass{V}$ and $\bmass{V}$ are the total mass and the total first variation 
measure, respectively, of $V$.
\end{enumerate}
Furthermore, 
if~\eqref{hypersurface1}-\eqref{hypersurface4} fail and if no component
of $\partial N$ is a minimal surface,
then 
the interior of $N$ contains a smooth, stable, two-sided closed
minimal hypersurface that  is an embedded submanifold
or the double cover of an embedded submanifold.
\end{theorem}

Two-sidedness of $M$ means (by definition) that the normal bundle is orientable, i.e, that $M$ has
a continuous unit normal vectorfield.  If the ambient space $N$ is orientable, then two-sidedness
of $M$ is equivalent to orientability of $M$.

\begin{remark}\label{SingularRemark}
Theorem~\ref{HypersurfaceCase} remains true (with the same proof) 
for  $n\ge 7$ provided  ``smooth''  
is replaced by ``smooth except for a singular set of Hausdorff dimension 
at most  $n-7$"
in statements~\eqref{hypersurface1},
 \eqref{hypersurface2},
 \eqref{hypersurface3},
and in the ``furthermore" assertion.  
\end{remark}

For general codimensions (and without assuming
any mean convexity or boundary regularity of $N$), we have:

\begin{theorem}\label{GeneralCodimension}
Suppose $N$ is a compact subset of an $(n+1)$-dimensional Riemannian
manifold and that $k\le n$.  The following are equivalent:
\begin{enumerate}[\upshape (a)]
\item\label{General1}
 The space $N$ contains no nonzero, stationary, $k$-dimensional varifolds
\item\label{General2}
 There is an increasing function $\phi: \RR^+\to \RR^+$ with $\phi(0)=0$
such that
\[ 
     \mass{V} \le \phi (\bmass{V})
\]
for every $k$-dimensional varifold $V$ in $N$ with $\bmass{V} < \infty$.
\item\label{General3}
 There is a constant $c<\infty$ such that
\[
   \mass{V}  \le c\, \bmass{V}
\]
for every $k$-dimensional varifold $V$ in $N$.
\end{enumerate}
\end{theorem}

\begin{proof}[Proof of Theorem~\ref{GeneralCodimension}]
Each statement is a special case of the following statement.   
Hence it suffices to show
that statement~\eqref{General1} implies
 statement~\eqref{General3}.   
 We do this by assuming that statement~\eqref{General3} fails and showing 
 that statement~\eqref{General1} must then also fail.

Failure of statement~\eqref{General3} 
means that there is a sequence $V_i$ of $k$-dimensional varifolds in $N$ with
\begin{equation}\label{ratio}
     \frac{\mass{V_i}}{\bmass{V_i}} \to \infty
\end{equation}
We may assume that $\mass{V_i} \equiv 1$, since otherwise we can replace $V_i$ with
$V_i/|V_i|$.
Note that $\bmass{V_i}\to 0$ by~\eqref{ratio}.
By compactness (Theorem~\ref{CompactnessTheorem}), 
a subsequence of the $V_i$ converges to a limit varifold $V$ with 
  $\mass{V}=1$
and with $\bmass{V}=0$, which violates~\eqref{General1}.
\end{proof}

Fix a $k$ and let $c_N$ be the supremum (possibly infinite) of $\frac{|V|}{|\delta V|}$
among nonzero $k$-dimensional varifolds $V$ in $N$.  Thus $c_N$, if finite, is the
best constant in 
the inequality~\ref{GeneralCodimension}\eqref{General3}.   
The reader may enjoy, as an exercise,
proving that the supremum is attained, and that $c_N$ is an upper semicontinous function of $N$ with respect to the Hausdorff metric on 
the space of compact subsets of the the Riemannian $(n+1)$-manifold
$N^*$ that contains $N$.  (The constant $c_N$ also depends upper semicontinuously
on the riemmanian metric on $N$.)

\begin{corollary}\label{cor:SpecialCase}
Suppose $N$ is a compact subset of a Riemannian manifold and that $N$
contains no nonzero, stationary, $k$-dimensional varifolds.  Then
there is a $c<\infty$ such that 
\begin{equation}\label{e:SpecialCase}
     |M| \le c \left( |\partial M| + \int_M |H|\,dA\right).
\end{equation}
whenever $M\subset N$ is a compact, smoothly immersed $k$-dimensional manifold.
\end{corollary}

\begin{proof}
The corollary follows immediately because if $V$ is the varifold associated
to $M$, then the left sides of~\ref{GeneralCodimension}\eqref{General3}
and~\eqref{e:SpecialCase} are equal, and the right side 
 of~\ref{GeneralCodimension}\eqref{General3} is less than or equal
 to the right side of~\eqref{e:SpecialCase}.
   (See Theorem~\ref{dVTheorem}.  
   If $M$ is embedded, then the right sides of~\ref{GeneralCodimension}\eqref{General3}
   and~\eqref{e:SpecialCase} are equal.)
\end{proof}

\begin{proof}[Proof of Theorem \ref{HypersurfaceCase}]
Each of the statements
  \eqref{hypersurface1}--\eqref{hypersurface4}
 is a special case of the following statement, so to prove
 their equivalence, it suffices to assume that statement~\eqref{hypersurface4}
fails and show that statement~\eqref{hypersurface1} must then also fail.   
 By Theorem~\ref{GeneralCodimension}, $N$ must contain a nonzero
$n$-dimensional stationary varifold $V$.  We must show that $N$ also contains
a  minimal embedded hypersurface with small singular set 
(in particular, with empty singular set if $n<7$).   That implication,
as well as the last assertion of
Theorem~\ref{HypersurfaceCase}, is given by
the following theorem.
\end{proof}

\begin{theorem}\label{FlowTheorem}
Suppose that $N$ is a compact, 
connected, mean-convex $(n+1)$-dimensional Riemannian manifold with smooth,
nonempty boundary, and that no connected component of $\partial N$ is a minimal surface.
 Suppose also that $N$ contains a nonzero
stationary $n$-dimensional varifold.   Then the interior of $N$ contains a closed minimal hypersurface 
 $M$ 
such that
\begin{enumerate}[\upshape (i)]
 \item\label{flow1} $M$ is a smooth embedded submanifold except
     for a closed singular set of Hausdorff dimension at 
  most $n-7$,
  \item\label{flowstable} the smooth part of $M$ is stable, and 
        each one-sided connected component of the smooth part of $M$
        has a stable, two-sided double cover.
  \item\label{invariant} $M$ is invariant under the isometry group of $N$.
\end{enumerate}
\end{theorem}

\begin{proof}
Let $V$ be the stationary varifold.
If  $\spt\|V\|$ (the spatial support of $V$) touched $\partial N$, then by the strong maximum principle (Theorem~\ref{MaximumPrinciple}),
$\spt\|V\|$ would contain an entire connected component of $\partial N$, and that component would 
have to be a smooth minimal surface, contrary to the hypotheses. 
Thus $\spt\|V\|$ does not touch $\partial N$.

Now we use properties of the mean curvature flow proved in 
\cite{WhiteMeanConvex1}*{\S 11}.
Let $M_t$ ($t\ge 0$) be the mean curvature flow with $M_0=\partial N$.  
Let $K_t$ be the closed region bounded by $M_t$.
Since $K_0=N$ is mean convex and since no component of $\partial K_0$ is minimal, the $K_t$ are nested and mean convex.
Furthermore, the mean curvature does not vanish at any regular point of $M_t$
for $t>0$.

Since $V$ is a stationary varifold, and since $M_t$ 
and $\spt\|V\|$ are disjoint at time $t=0$, they must be
disjoint at all times  by the avoidance principle for mean curvature flow
(see Theorem~\ref{AvoidanceTheorem}).
That is, $\spt\|V\| \subset K_t$ for all $t$.  Consequently, $K:=\cap_tK_t$ is nonempty.  

For simplicity, assume that $K$ is connected. (Otherwise argue
as below for each of the connected components of $K$.)
According to~\cite{WhiteMeanConvex1}*{\S11}, $\partial K$ is an embedded minimal hypersurface with
singular set of dimension at most $n-7$.
Furthermore, 
the $M_t$ converge to a limit $\tilde M$, where $\tilde M=\partial K$ if $K$ has nonempty interior, and
    $\tilde M$ is a double cover of $\partial K =K$ if 
 $K$ has empty interior.   In either case, the convergence
is smooth away from the singular set of $\partial K$.  The $M_t$ are two-sided (orient the normal bundle by the mean curvature vector), and therefore $\tilde M$ must also be two-sided.

Furthermore, since the $M_t$ have nonzero mean curvature vector pointing
toward $M$ and since they foliate
one side of $\tilde M$ (away from the singular set of $M$), $\tilde M$ must be stable. 
(Indeed, as explained in \cite{WhiteMeanConvex1}*{3.5}, $\tilde M$ has a one-sided minimizing property
that is stronger than stability.)

Invariance of $K$ and therefore of $M=\partial K$ under the isometry group of $N$ is 
clear from the construction.
\end{proof}

\begin{remark}
One can also prove Theorem~\ref{FlowTheorem} without using mean curvature flow.   Roughly speaking,
one obtains $M$ by minimizing area among hypersurfaces that enclose $\spt\|V\|$.   
More precisely, one first shows that for any sufficiently small $\delta>0$, there is a open set
$U= U_\delta$ containing $\spt\|V\|$ that minimizes 
  $|\partial U| - \delta |U|$ among all such open sets.   
(Some work is required to show that in a minimizing sequence of open sets, the boundaries
can be kept away from $\spt\|V\|$.)
  Such a surface is smooth except for
  a singular set of dimension at most $n-7$.   One then gets $M$ as a subsequential  limit 
  as $\delta\to 0$ of the varifolds associated to the $\partial U_\delta$.
  
  To get a $G$-invariant $M$ (where $G$ is the isometry group of $N$),
  one considers only those sets $U$ that are $G$-invariant.
\end{remark}

\begin{remark}
In Theorem~\ref{HypersurfaceCase}, we assumed that $\partial N$ was nonempty.
In the case $\partial N=\emptyset$, 
Schoen and Simon~\cite{SchoenSimonStable}, using the work of Pitts~\cite{PittsBook},
proved that $N$ must contain a closed, embedded
hypersurface with singular set of dimension at most $n-7$.
Thus the statements~\eqref{hypersurface1}--\eqref{hypersurface4}
in Theorem~\ref{HypersurfaceCase} 
(as modified in Remark~\ref{SingularRemark}) all fail for such $N$.
\end{remark}

\begin{remark}\label{OpenQuestion}
It would be interesting either to prove an analog of Theorem~\ref{FlowTheorem}
for surfaces of codimension $>1$, or to construct counterexamples.  In particular, 
let $N$ be a compact, $k$-convex Riemannian manifold $N$.  
(We say that $N$ is $k$-convex if  $\kk_1+\dots +\kk_k\ge 0$ at each boundary point,
where $\kk_1\le \kk_2\le \dots \le \kk_n$ are the principal curvatures
of $\partial N$ with respect to the inward unit normal.)
 Suppose  that $N$ contains
a nonzero stationary $k$-varifold.   Must $N$ then also contain a stationary
$k$-varifold with some additional properties?
  For example, must it contain an {\em integral} stationary 
$k$-dimensional varifold?  If so, must it contain a closed $k$-dimensional
minimal submanifold with a small singular set?   A positive answer would imply
an extension of the main theorem, Theorem~\ref{HypersurfaceCase}, from hypersurfaces to $k$-dimensional
surfaces.
\end{remark}

\section{Examples}\label{examples}

\begin{theorem}\label{RicciTheorem}
Suppose that $N$ is a compact, connected, mean-convex Riemannian
manifold with smooth, nonempty boundary, and that no connected component
of $\partial N$ is a minimal surface.
Suppose also that the dimension of $N$ is at most $7$ and that
the Ricci curvature of $N$ is everywhere positive.
Then the isoperimetric inequalities listed in Theorem~\ref{HypersurfaceCase} hold.

More generally, if $N$ has nonnegative Ricci curvature, then those isoperimetric
inequalities hold unless $N$ contains a closed, embedded, totally geodesic
hypersurface $M$ such that $\text{Ric}(\nu,\nu)=0$ for every unit normal $\nu$ to $M$.
\end{theorem}

\begin{proof}
Suppose $N$ has positive Ricci curvature and that those isoperimetric inequalities fail.
Then by Theorem~\ref{HypersurfaceCase}, $N$ contains a 
smooth, stable, two-sided minimal hypersurface $M$.
As is well-known~\cite{SchoenSurvey}*{\S5}, such a surface is incompatible with
positive Ricci curvature for the following reason.
The stability of $M$ means that
\begin{equation}\label{stability}
   \int_M \left ( \text{Ric}(\nu,\nu) + |A|^2\right) |f|^2  \le \int_M |\nabla f|^2
\end{equation}
for all smooth functions $f:M \to \RR$, where $\nu$ is a 
unit normal vectorfield to $M$
and $A$ is the second fundamental form of $M$.
But if we set $f\equiv 1$, the left side of~\eqref{stability} is positive and the right side is $0$, a contradiction.

If $N$ has nonnegative curvature and if the isoperimetric inequalities fail, then (by letting $f\equiv 1$
in~\eqref{stability}) we see that $|A|\equiv 0$ (i.e., that $M$ is totally geodesic)
and that $\text{Ric}(\nu,\nu)\equiv 0$.
\end{proof}

In recent years there have been a number of investigations of minimal
surfaces in ambient spaces of the form $M\times \RR$.  
(See for example 
\cite{HoffmanLiraRosenberg},
\cite{Rosenberg2002},
 \cite{NelliRosenberg}, 
 \cite{MeeksRosenberg2004}, and
 \cite{MeeksRosenberg2005}.)
 Note that $M\times \RR$ is foliated by the minimal surfaces $M\times \{z\}$.
 Using Theorem~\ref{GeneralCodimension}, 
we can prove isoperimetric inequalities in very general compact subsets
of ambient spaces admitting such foliations:

\begin{theorem}\label{ProductTheorem}
Let $N^*$ be an $(n+1)$-dimensional Riemannian manifold.
Let $f:N^*\to \RR$ be a smooth function with nowhere vanishing gradient such that
the level sets of $f$ are minimal hypersurfaces or, more generally, such that 
the sublevel sets $\{ x : f(x)\le z \}$ are mean convex.
Let $N$ be a compact subset of $N^*$ such that
for each $z\in \RR$, no connected component of $f^{-1}(z)$ is a minimal hypersurface
lying in entirely in $N$. 
 Then the  isoperimetric 
inequalities in Theorem~\ref{GeneralCodimension}
and in Corollary~\ref{cor:SpecialCase} hold for $k=n$.
\end{theorem}

Of course if $N$ does contain a connected component $M$ of $f^{-1}(z)$ that
is a minimal hypersurface, then that component must be a compact minimal
hypersurface without boundary, and thus
all of the isoperimetric inequalities fail.   (They all fail for $M$.)

\begin{proof}
Although this theorem is about hypersurfaces, we cannot
apply Theorem~\ref{HypersurfaceCase} because
we are not making any assumptions about $\partial N$. 
However, by Theorem~\ref{GeneralCodimension}, it suffices
to show that $N$ contains no nonzero stationary $n$-varifolds.

Suppose to the contrary that $V$ is such a varifold.
Since $N$ is compact, $\spt\|V\|$ (the spatial
 support of $V$) is compact,
and thus the function $f$ has a maximum value
$z$ on $\spt\|V\|$.  Hence $\spt\|V\|$ touches but lies on the $\{f\le z\}$ side of the smooth
minimal hypersurface $M=f^{-1}(z)$.  
By the strong maximum principle (Theorem~\ref{MaximumPrinciple}), $\spt\|V\|$
must contain an entire component of  $M$ and that component must be a minimal hypersurface.
But by hypothesis, $N$ does not contain any such component.
\end{proof}

\section{Minimal Disks in a $3$-Manifold}\label{disks}

Suppose $N$ is a compact Riemannian $3$-manifold with smooth, mean-convex
boundary.   Often embedded minimal disks in $N$ with boundary curves in $\partial N$ are
of particular interest.   Could $N$ admit isoperimetric inequalities for such disks even
if it did not admit isoperimetric inequalities for other kinds of minimal surfaces?
The answer is, in general, no:

\begin{theorem}\label{DiskTheorem}
Suppose $N$ is a compact, mean convex Riemannian $3$-manifold whose
boundary is a sphere with nowhere vanishing mean curvature.
Then one and only one of the following holds:
\begin{enumerate}[\upshape (a)]
\item\label{TopologicalBall}
$N$ is diffeomorphic to a ball, and the isoperimetric inequalities listed in Theorem~\ref{HypersurfaceCase} hold.
\item\label{BigDisks}
There is a compact family $\mathcal{C}$ of smooth embedded curves in $\Sigma$ such
that
\begin{equation*}
       \sup \area(D) = \infty,
\end{equation*}
the $\sup$ being over all smooth, embedded minimal disks $D$ in $N$ with 
   $\partial D\in \mathcal{C}$.
 \end{enumerate}
\end{theorem}

Note that  \eqref{BigDisks} implies that there is {\it no} bound
of the form
\begin{equation}\label{AnyBound}
     |D| \le \phi\left( |\partial D| +  \|\partial D\|_{C^k} + 
      \sup_{x,y\in \partial D}\frac{\dist_{\partial D}(x,y)}{\dist_N(x,y)}
      \right)
\end{equation}
for embedded minimal disks $D$ with $\partial D\subset \partial N$.
In other words, there is no $k$ and no increasing function $\phi:\RR\to \RR$
for which~\eqref{AnyBound} holds.

\begin{proof}  Clearly~\eqref{BigDisks} violates the isoperimetric inequalities, so 
\eqref{BigDisks} implies failure of~\eqref{TopologicalBall}.

Now suppose that~\eqref{TopologicalBall} fails.
If the isoperimetric inequalities fail, then $N$
contains a smooth, closed minimal surface by 
Theorem~\ref{HypersurfaceCase}.
Thus either $N$ is not diffeomorphic to a ball, or
$N$ contains a closed minimal surface.  
According to \cite{WhiteNewApplications}*{\S3}, in either case
we have~\eqref{BigDisks}.
\end{proof}

\section{Nonsmooth Boundaries}\label{nonsmooth}

In Theorems~\ref{HypersurfaceCase}, 
 \ref{FlowTheorem}, and~\ref{RicciTheorem},
 $\partial N$ was assumed
to be smooth.  
Sometimes it is convenient to work in domains $N$ with boundaries
that are only piecewise smooth.  In fact, the theorems are true under
very mild boundary regularity and mean convexity assumptions.

Suppose $N$ is an arbitrary compact, connected subset of an Riemannian 
 $(n+1)$-manifold.  
Examination of the proofs of 
  Theorems~\ref{HypersurfaceCase}, 
 \ref{FlowTheorem}, and~\ref{RicciTheorem}
show that they remain valid for $N$ provided the 
$N$ has the following two properties:
\begin{enumerate}
  \item\label{Property1}
   If $V$ is a stationary $n$-varifold in $N$, then $\spt\|V\|$
   is contained in the interior of $N$.
  \item\label{Property2}
   If $X$ is a compact subset of the interior of $N$, then there is
  a strictly mean convex set $K_0$ with smooth boundary such that $X$ is contained
  in the interior of $K_0$.  (One uses $\partial K_0$ as the      
  initial surface  for the 
  mean curvature flow in the proof of Theorem~\ref{FlowTheorem}, 
  with $X=\spt\|V\|$.)
\end{enumerate}

Suppose, for example, that $\partial N$ is the a union of smooth
$n$-manifolds-with-boundary that meet in pairs along common edges
at interior angles that are everywhere strictly between $0$ and $\pi$.   
Suppose also that no connected component of $\partial N$ is a smooth
minimal surface.
 Then $N$ has 
 properties~\eqref{Property1} and~\eqref{Property2},
 and therefore 
 Theorems~\ref{HypersurfaceCase}, 
 \ref{FlowTheorem}, and~\ref{RicciTheorem}
 hold for $N$.
Property (1) follows easily from the 
maximum principle~\ref{MaximumPrinciple}.
The $K_0$ of property (2) is obtained by rounding off the corners
of $N$.  (The rounding may be accomplished as follows.
Let $K^*$
be the union of all closed geodesic balls in $N$ of radius $\eps$.
For small enough $\eps$, the boundary of $K^*$ will be $C^{1,1}$,
and under mean curvature flow it will immediately move into the
interior of $N$ and be smooth with everywhere positive mean curvature.)

    
\section{Nonlinear Inequalities}\label{NonlinearSection}

We now consider isoperimetric inequalities of the form
\begin{equation}\label{nonlinear}
    |M|^{1 - 1/k} \le c \left( |\partial M| + \int_M |H|\,dA \right)
\end{equation}
for $k$-dimensional surfaces $M$, or the corresponding inequalities
\begin{equation}\label{nonlinearvarifold}
   |V|^{1 - 1/k} \le c\, |\delta V|
\end{equation}
for $k$-varifolds $V$.

First we observe that that the inequality~\eqref{nonlinearvarifold} can never be valid in any ambient
manifold (even Euclidean space) 
 if we allow arbitrary varifolds.
For suppose $V=V_\eps$ is $\eps$ times the varifold associated
to a smooth $k$-manifold in $N$.
Then the left and right sides of~\eqref{nonlinearvarifold}  are proportional to $\eps^{1-1/k}$
and $\eps$, respectively, so the inequality necessarily fails for small $\eps$.

On the other hand, Allard proved that~\eqref{nonlinearvarifold} is true in Euclidean space if we require that $V$ be a finite-mass
integer-multiplicity rectifiable $k$-varifold or, more generally, that $V$ be a finite-mass
rectifiable $k$-varifold with density $\ge 1$ almost everywhere:
see~\cite{AllardFirstVariation}*{7.1} or~\cite{SimonBook}*{18.6}.
In particular, we have~\eqref{nonlinear} for any $k$-manifold $M$ with $|M|$ finite.

For general ambient manifolds we have the following theorem

\begin{theorem}\label{NonlinearTheorem}
Suppose $N$ is a compact region in a Riemannian manifold.  Then there are constants
$\alpha>0$ (depending on $N$) and $c'=c'_k<\infty$ (depending only on~$k$) such that if $M$ is a smooth $m$-dimensional
surface in $N$
with $|M|\le \alpha$, then
\begin{equation}\label{nonlinearagain}
   |M|^{1 - 1/k} \le c'\left( |\partial M| + \int_M | H | \right).
\end{equation}
More generally, if $V$ is a rectifiable $k$-varifold in $N$ 
with density $\ge 1$ almost everywhere and with $|V|\le \alpha$, then
\begin{equation}\label{nonlinearvarifoldagain}
  |V|^{1 - 1/k} \le c' | \delta V|.
\end{equation}
\end{theorem}

\begin{proof}  We describe the proof of~\eqref{nonlinearagain}; the proof 
of~\eqref{nonlinearvarifoldagain} is essentially the same.
Embed the Riemannian manifold isometrically in a Euclidean space $E$.   Let $H_e$ be the mean curvature
of $M$ as a submanifold of $E$.  Then
\begin{equation}\label{2ndForm}
    |H_e| \le |H| + K
\end{equation}
holds at each point for some $K$ depending only on the second fundamental form
of $N$ at that point.   In particular, since $N$ is compact there is a constant $K$
(independent of $M$) such that~\eqref{2ndForm} holds at all points.

By the isoperimetric inequality in Euclidean space, we have
have
\[
   |M|^{1-1/k} \le c \left( |\partial M| + \int |H_e|\, dA \right)
             \le c \left( |\partial M| + \int |H|\, dA + K |M| \right).
\]
Thus
\[
   |M|^{1-1/k} ( 1 -  c  K |M|^{1/k} ) 
       \le 
    c_m\left( |\partial M| + \int |H|\, dA  \right).
\]
Thus for $|M|\le \alpha$,
\[
  |M|^{1-1/k} ( 1 - c K \alpha^{1/k} ) 
     \le
    c  \left( |\partial M| + \int |H|\, dA  \right).
\]
Now we let $\alpha =  (2cK)^{-k}$ and $c'= 2c$.
\end{proof}    

\begin{remark}
Theorem~\ref{NonlinearTheorem}
 remains true (with the same proof) for noncompact $N\subset E$ provided the
norm of second fundamental form of $N$ is bounded.
\end{remark}

It remains to consider whether the nonlinear isoperimetric 
inequality~\eqref{nonlinearagain} holds (perhaps with a worse constant)
for $M$ with $|M|\ge \alpha$.   Here we consider the case of $n$-dimensional surfaces
$M$ in a compact, mean-convex Riemannian $(n+1)$-manifold $N$ with smooth (or piecewise smooth)
boundary.   (Exactly the same reasoning applies to rectifiable varifolds with density bounded below by $1$ almost
everywhere.) 
The answer is then simple: a suitable constant $c'$ exists if and only if $N$ contains
no closed minimal hypersurface.  
 The ``only if'' is immediate since a closed minimal hypersurface
would be a counterexample to the inequality.   Thus suppose $N$ contains no closed minimal
hypersurface.  Then by Theorem~\ref{HypersurfaceCase}, we have the linear inequality
\[
    |M| \le c \left( |\partial M| + \int |H|\,dA\right)
\]
for all $M$.   If $|M|\ge \alpha$, then $|M|^{1-1/n}\le \alpha^{-1/n}|M|$, so we get the nonlinear
inequality~\eqref{nonlinearagain} with constant $c' = c\alpha^{-1/n}$.

\section{Appendix: Varifolds}

Let $N$ be a compact subset of a Riemannian $(n+1)$-manifold $N^*$.
Let
\[
  G_k(N) = \{ (p,S): \text{$p\in N$ and $S$ is a $k$-dimensional subspace of $\Tan_p(N^*)$} \}.
\]

\begin{definition}
A $k$-dimensional varifold (or $k$-varifold, for short) in $N$ is a finite Borel measure on $G_k(N)$.
\end{definition}

(This definition and some of the discussion below need to be modified slightly for noncompact $N$.
However, in this paper $N$ is always compact.)

Every $C^1$ embedded, compact $k$-dimensional submanifold $M$ (with or without boundary)
in $N$ determines a varifold $V=\vv(M)$ (called the varifold associated to $M$) by setting
\begin{equation*}
   V(U) = \area \{ p\in M:  (p, \Tan_pM) \in U \}
 \end{equation*}
 for every borel set $U\subset G_k(N)$.
Consequently $k$-varifolds can be regarded as generalized $k$-dimensional surfaces.  

One can also define, in a similar way, a varifold $V=\vv(M)$ associated to
a $C^1$ immersed submanifold with boundary.  Indeed, if 
 $\iota: M \to N$ is an immersion, we let
 \begin{equation*}
     V(U) = \area \{ p\in M: (\iota(p),\iota_\#( \Tan_pM)) \in U \}
  \end{equation*}
  where area is with respect to the induced metric on $M$.

If $V$ is a $k$-varifold, we let $\mass{V}$ denote its mass:
\[
  \mass{V} = V(G_k(N)).
\]
The mass
of $V$ is also written (for reasons that need not concern us here) as $\|V\|(N)$.
Note that if $V=\vv(M)$, then the mass $\mass{V}$ of $V$ is just the area of $M$.  

Since $V$ is a measure on $G_k(N)$, its support is a closed subset
of $G_k(N)$.  If we project that closed set to $N$ by the projection
$(x,S)\mapsto x$, then result
is the {\it spatial support} of $V$, written $\spt \|V\|$.  It is the smallest closed subset $K$ of $N$
such that
\[
     V \{ (x,S) \in G_k(N):  x\notin K\}  = 0.
\]

Now suppose $X$ is a smooth vectorfield on $N^*$.
For a smoothly embedded, compact submanifold $M$, 
one
 has 
\begin{equation}\label{DivergenceTheorem} 
   \int_{M} \Div^MX = \int_{\partial M}  X\cdot \nu - \int_M H\cdot X 
\end{equation}
where $H(p)$ is mean curvature of $M\subset N$ at $p$, where
$\nu(p)\in \Tan_pM$ is the outward pointing unit normal to $\partial M$ at $p$, and
where $\Div^MX(p)$ is the 
divergence of $X$ over $M$ at $p$.  That is,
\[
  \Div^MX(p) = \sum_{i=1}^k (\nabla_{e(i)}X) \cdot e(i)
\]
where $e(1),\dots, e(k)$ is an orthonormal basis for $\Tan_pM$.

(To prove~\eqref{DivergenceTheorem}, one breaks $X$ into tangential and normal 
parts: see the proof of~\cite{SimonBook}*{9.6}.
Incidentally, the quantity~\eqref{DivergenceTheorem} 
occurs in the first variation formula for area: it is equal to the initial rate of change of area
of any one-parameter family of surfaces starting at $M$ and moving with initial velocity~$X$.
See~\cite{SimonBook}*{\S 9}.)

\begin{theorem}\label{EasyVersion}
Suppose that  $M$ is a compact $m$-dimensional surface in Euclidean space.  Then
\begin{equation}\label{LinearInequalityAgain}
    |M| \le \frac{r}m \left( |\partial M| + \int_M |H| \right)
\end{equation}
where $r$ is the radius of the smallest ball containing $M$.
\end{theorem}

\begin{proof}
We may assume that the smallest ball containing $M$ is centered at the origin.
Let $X(x)=x$.  Then $\Div^MX(p) \equiv m$, 
 so~\eqref{LinearInequalityAgain} follows immediately
from~\eqref{DivergenceTheorem}.
\end{proof}

The left side of~\eqref{DivergenceTheorem} can be generalized to an  arbitrary $k$-varifold in $N$ as follows.
We define a linear functional $\delta V$ on the space of smooth tangent vectorfields $X$ on $N^*$ by
\[
  \delta V(X) = \int \Div_S X(p) \, dV(p,S),
\]
the integral being over all $(p,S)\in G_k(N)$.  Here 
\[
  \Div_S X(p) = \sum_{i=1}^k (\nabla_{e(i)}X)\cdot e(i)
\]
where $e(1),\dots, e(k)$ is an orthonormal basis for $S$.

The linear functional $\delta V$ has a norm $|\delta V| \in [0,\infty]$ given by
\[
   \bmass{V} = \sup_X \delta V(X)
\]
where the $\sup$ is over all smooth tangent vectorfields $X$ on $N^*$ 
such that $|X(p)|\le 1$
for all $p$.  Equivalently, $\bmass{V}$ is the smallest number $c \in [0,\infty]$ such that
\[
   \delta V(X) \le c \, \|X\|_0
\]
for every smooth vectorfield $X$, where $\|X\|_0$ is the $C^0$ norm
(i.e., the $\sup$ norm)  of $X$.

(The norm $\bmass{V}$ would be written $\| \delta V\|(N)$
in the notation of \cite{AllardFirstVariation} and~\cite{SimonBook}.)

Although $\delta V(X)$ is finite for every varifold $V$ and $C^1$ vectorfield $X$,
the norm $|\delta V|$ may be infinite.

\begin{theorem}\label{dVTheorem}
 Let $V=\vv(M)$ be the varifold associated
to a smoothly immersed manifold $M$  in $N$.  Then
\begin{equation}\label{ImmersedInequality}
     |\delta V| \le |\partial M| + \int_M|H|.
\end{equation}
If $M$ is embedded, then equality holds.
\end{theorem}

\begin{proof} 
Note by~\eqref{DivergenceTheorem}, which also holds for immersed
surfaces, we have
\begin{equation}\label{dVX}
   | \delta V(X)| \le \left( |\partial M| + \int_M |H|\right) \|X\|_0
\end{equation}
which implies~\eqref{ImmersedInequality}.  In the embedded case, one can choose $X$
with $\|X\|_0\le 1$ so that $X=\nu$ on 
   $\partial M$, 
and so that, except for a small set in $M$, 
$X= -H/|H|$ provided $H\ne 0$.  Then $\delta V(X)$ will be arbitrarily
 close to the
right side of~\eqref{ImmersedInequality}.
By definition of $\bmass{V}$, this implies that the 
 left side of~\eqref{dVX} is greater than or equal to the right side,
 and thus that the two sides must be equal.
(In the immersed case, this choice of $X$
is not always possible because $X$ must be well-defined on $N$.)
\end{proof}

\begin{definition}
A varifold $V$ is called {\em stationary} provided $\delta V=0$ (or, equivalently, provided
 $\bmass{V}=0$.)
\end{definition}

For a smoothly embedded surface $M$, the associated varifold is stationary
if and only if $M$ is a minimal surface without boundary 
(by Theorem~\ref{dVTheorem}).
Thus stationary varifolds are generalizations of minimal surfaces without
boundary.

\begin{theorem}\label{CompactnessTheorem}
Let $V_i$ be a sequence of $k$-dimensional varifolds in $N$.  Suppose that $N$ is compact
and  that $\sup \mass{V_i} <\infty$.
Then there is a subsequence $V_i'$ that converges to a varifold $V$. Furthermore,
\begin{equation}\label{MassesConverge}
     \mass{V} = \lim \mass{V_i'}
\end{equation}
and
\begin{equation}\label{LowerSemicontinuity}
   \bmass{V}  \le \liminf \bmass{V_i'}.
 \end{equation}
\end{theorem}

Here convergence of varifolds means weak convergence of measures.

\begin{proof}
The existence of a convergent subsequence follows from the Riesz Representation Theorem
and the Banach-Alaoglu Theorem.  
Continuity of mass~\eqref{MassesConverge} is an immediate consequence of weak convergence
(since the space $G_k(N)$ is compact.)

Note that $\delta V(X)$ is the integral with respect to the measure $V$ of the continuous function
\begin{align*}
    G_k(N)  &\to \RR \\
    (p,S) &\mapsto  \Div_S X(p).
\end{align*}
 Consequently, by definition of weak convergence of measures,
\begin{equation}\label{weak}
    \delta V(X) = \lim \delta V_i'(X). 
\end{equation}                      
But
\[
    \delta V_i' (X) \le \bmass{V_i'}\cdot \|X\|_0,
\]
so from~\eqref{weak} we see that
\[
   \delta V(X) \le \liminf \bmass{V_i'} \cdot \|X\|_0.
\]
Taking the supremum over all $X$ with $\|X\|_0\le 1$ gives~\eqref{LowerSemicontinuity}.
\end{proof}

\begin{theorem}[Maximum Principle]\label{MaximumPrinciple}
Let $B$ be an open set in a Riemannian $(n+1)$-manifold $N^*$.
Let $M$ be a smooth, connected hypersurface properly embedded in $B$ and dividing $B$
into two components.
Let $\Omega$ be one of the two components of $B\setminus M$.

Suppose that $\Omega$ is mean concave along $M$, i.e., that 
at each point of $M$, the mean curvature is a nonnegative
multiple of the outward unit normal to $\Omega$.

Suppose $S$ is the spatial support of a nonzero stationary $n$-varifold in $N^*$ such that $S$ 
is disjoint from $\Omega$.
If $S$ contains any point of $M$, 
then it must contain all of $M$,
and $M$ must be a minimal surface.
\end{theorem}

See~\cite{SolomonWhite} for the proof.  
(Note the additional remarks at the end of that paper.)  

For a full treatment of varifolds, see~Allard's paper \cite{AllardFirstVariation}
or chapter~8 of Simon's book~\cite{SimonBook}.
Simon's book only considers varifolds in Euclidean space.
However, as explained in~\cite{AllardFirstVariation}*{\S 4.4},
the study of $k$-varifolds in a Riemannian manifold $N$ can be reduced
to the Euclidean case by  isometrically embedding $N$ into a Euclidean
space $E$.  If one does that, the terminology in this paper needs to be
interpreted accordingly.
For example, ``$V$ is a $k$-varifold in $N$'' should be interpreted as:
$V$ is a $k$-varifold in $E$ and $V$ vanishes outside of 
 \[
       \{ (x,S)\in G_k(E): \text{$x \in N$ and $S\subset \Tan_xN$}  \}.
 \]
  Also the mean curvature
$H$ of a submanifold $M\subset N$ means the component of the mean curvature of $M\subset E$ that is tangent to $N$.   And the norm $\bmass{V}$ (which
should perhaps be written $\bmass{V}_N$) means the supremum of 
$\delta V(X)$ over all $C^1$ vectorfields $X$ with $\|X\|_0\le 1$ that are tangent
to $N$ at each point of $N$.

We conclude this paper by proving that hypersurfaces  moving by 
mean curvature flow cannot collide with stationary varifolds. (This was used in
the proof of Theorem~\ref{FlowTheorem}.)   We begin with the smooth case:

\begin{proposition}\label{AvoidanceProposition}
Let $S$ be the spatial support of a nonzero 
stationary $n$-varifold in a Riemannian $(n+1)$-manifold $N$.
Let $t\in [0,T]\mapsto M_t$ be a mean curvature flow, where the $M_t$'s are compact
smooth embedded hypersurfaces in $N$.
If $M_t$ and $S$ are disjoint at time $0$, then they must remain disjoint for all $t\in [0,T]$.
\end{proposition}

\begin{proof}
Fix  a small $\eps$ with $0<\eps< \dist(M_0, S)$ and a small $\delta>0$.
(How small will be specified shortly.)  
 Let
  $K_0= K(\eps, \delta)$ be the set of points in $N$ at distance $\le \eps$ from $M_0$.
Note that $K_0$ is disjoint from $S$.
Let $K_t$, $t\in [0,T]$ be the result of letting $K_0$ evolve so that is boundary moves
with velocity $H - \delta \nu$, where $\nu$ is the outward unit normal to $K_t$.  We choose $\eps$
and $\delta$ sufficiently small that $\partial K_t$ is smooth for all $t\in [0,T]$.
By choosing $\delta$ small enough (after having fixed $\eps$), we can also ensure
that $M_t$ lies in the interior of $K_t$ for all $t\in [0,T]$.

We claim that $K_t$ is disjoint from $S$ for all $t\in [0,T]$.  
For suppose not.
Thus $K_t$ will collide with $S$ at some first time $\tau\le T$.  At any point $p$
of contact of  $K_\tau$ and $S$, we have $H-\delta \nu\ge 0$, so in particular $H\cdot\nu>0$
at $p$, and therefore also on $B\cap \partial K_\tau$ for some small ball $B$ about $p$.
But that violates the maximum principle (Theorem~\ref{MaximumPrinciple}).
(Let the $M$ and $\Omega$ in that theorem be $B\cap \partial K_\tau$ and 
 $B\cap\, \text{interior}(K_\tau)$.)
Thus $K_t$ remains disjoint from  $S$  for all $t\in [0,T]$.
Since $M_t\subset K_t$, the surfaces $M_t$ also remain disjoint from $S$.
\end{proof}

\begin{theorem}\label{AvoidanceTheorem}
Let $S$ be the spatial support of a nonzero 
stationary $n$-varifold in a Riemannian $(n+1)$-manifold $N$.
Let $M_t$, $t\in [0,\infty)$, be 
the family of sets generated from $M_0$ by the level-set mean curvature flow in $N$.
If $M_0$ and $S$ are disjoint at time $0$, then they remain disjoint for all $t$.
\end{theorem}

In particular, the evolution of an initially mean convex hypersurface by mean curvature flow
is such a flow.

\begin{proof}
By Proposition~\ref{AvoidanceProposition}, $t\in [0,\infty) \mapsto S$ is a set-theoretic
subsolution of mean curvature flow  
(in the terminology of~\cite{IlmanenLevelSet} or~\cite{IlmanenBook}*{\S 10}) or a weak set-flow (in the 
terminology of~\cite{WhiteTopology}.)    Theorem~\ref{AvoidanceTheorem} is thus a special case of
the avoidance principle for set-theoretic subsolutions \cite{WhiteTopology}*{7.1}.
\end{proof}

\newcommand{\hide}[1]{}

\end{document}